\documentclass[10pt]{amsart}

\usepackage{eurosym}
\usepackage{graphicx,amscd,color,amsmath,amsfonts,amssymb,geometry,soul}

\newtheorem{theorem}{Theorem}
\theoremstyle{plain}

\newtheorem{definition}{Definition}
\newtheorem{example}{Example}

\newtheorem{lemma}{Lemma}

\newtheorem{proposition}{Proposition}
\newtheorem{remark}{Remark}

\numberwithin{equation}{section}
\begin{document}
\title[A note on multiple summing operators and applications]{A note on
multiple summing operators and applications}
\author[Albuquerque]{N. Albuquerque}
\address[N. Albuquerque]{Departamento de Matem\'{a}tica\\
\indent Universidade Federal da Para\'{\i}ba\\
\indent 58.051-900 - Jo\~{a}o Pessoa, Brazil.}
\email{ngalbqrq@gmail.com}
\author[Ara\'ujo]{G. Ara\'ujo}
\address[G. Ara\'ujo]{Departamento de An\'{a}lisis Matem\'{a}tico\\
\indent Facultad de Ciencias Matem\'{a}ticas\\
\indent Plaza de Ciencias 3\\
\indent Universidad Complutense de Madrid\\
\indent Madrid, 28040, Spain}
\email{gdasaraujo@gmail.com}
\author[Pellegrino]{D. Pellegrino}
\address[D. Pellegrino]{Departamento de Matem\'{a}tic,a\\
\indent Universidade Federal da Para\'{\i}ba\\
\indent 58.051-900 - Jo\~{a}o Pessoa, Brazil.}
\email{dmpellegrino@gmail.com and pellegrino@pq.cnpq.br}
\author[Rueda]{P. Rueda}
\address[P. Rueda]{Departamento de An\'{a}lisis Matem\'{a}tico\\
\indent Universidad de Valencia\\
\indent 46100 Burjassot, Valencia.}
\email{pilar.rueda@uv.es}
\thanks{N. Albuquerque, G. Ara\'{u}jo and D. Pellegrino are supported by
CNPq Grant 401735/2013-3 - PVE - Linha 2}
\thanks{P. Rueda is supported by Ministerio de Econom\'{\i}a y
Competitividad (Spain) MTM2011-22417.}

\begin{abstract}
We prove a new result on multiple summing operators and among other results
applications, we provide a new extension of Littlewood's $4/3$ inequality to
$m$-linear forms.
\end{abstract}

\maketitle



\section{Introduction}

Let $\mathbb{K}$ be the real scalar field $\mathbb{R}$ or the complex scalar
field $\mathbb{C}$. As usual, for a positive integer $N$ we define $%
\ell_\infty^N=\{(x_n)_{n=1}^\infty\subset \mathbb{K} \mbox{ bounded }\}$, $%
c_{0}=\left\{ \left( x_{n}\right)_{n=1}^\infty \subset \mathbb{K}:\lim
x_{n}=0\right\} $ and $e_{j}$ represents the canonical vector of $c_{0}$
with $1$ in the $j$-th coordinate and $0$ elsewhere. Littlewood's $4/3$
inequality \cite{Litt}, proved in 1930, asserts that
\begin{equation*}
\left( \sum\limits_{i,j=1}^{\infty }\left\vert U(e_{i},e_{j})\right\vert ^{%
\frac{4}{3}}\right) ^{\frac{3}{4}}\leq \sqrt{2}\left\Vert U\right\Vert
\end{equation*}%
for every continuous bilinear form $U:c_{0}\times c_{0}\rightarrow \mathbb{K}
$ or, equivalently,
\begin{equation*}
\left( \sum\limits_{i,j=1}^{N}\left\vert U(e_{i},e_{j})\right\vert ^{\frac{4%
}{3}}\right) ^{\frac{3}{4}}\leq \sqrt{2}\left\Vert U\right\Vert
\end{equation*}%
for every positive integer $N$ and all bilinear forms $U:\ell _{\infty
}^{N}\times \ell _{\infty }^{N}\rightarrow \mathbb{K}$.

It is well known that the exponent $4/3$ is optimal and it was recently
shown in \cite{diniz2} that the constant $\sqrt{2}$ is also optimal for real
scalars. For complex scalars, the constant $\sqrt{2}$ can be improved to $2/%
\sqrt{\pi }$, although it seems to be not known if this value is optimal.
The natural step further is to investigate sums
\begin{equation*}
\left( \sum\limits_{i_{1},\ldots ,i_{m}=1}^{N}\left\vert
U(e_{i_{^{1}}},\ldots ,e_{i_{m}})\right\vert ^{r}\right) ^{\frac{1}{r}}
\end{equation*}%
for $m$-linear forms $U:\ell _{\infty }^{N}\times \cdots \times \ell
_{\infty }^{N}\rightarrow \mathbb{K}$. The exponent $4/3$ need to be
increased to have a similar inequality for multilinear forms; this is what
the H.F. Bohnenblust and E. Hille discovered in 1931 (\cite{bh}, and also
\cite{sur}). More precisely, the Bohnenblust--Hille inequality asserts that
for every positive integer $m$ there is a constant $C_{m}\geq 1$ so that
\begin{equation}
\left( \sum\limits_{i_{1},\ldots ,i_{m}=1}^{N}\left\vert
U(e_{i_{^{1}}},\ldots ,e_{i_{m}})\right\vert ^{\frac{2m}{m+1}}\right) ^{%
\frac{m+1}{2m}}\leq C_{m}\left\Vert U\right\Vert  \label{lf}
\end{equation}%
for all positive integers $N$ and all $m$-linear forms $U:\ell _{\infty
}^{N}\times \cdots \times \ell _{\infty }^{N}\rightarrow \mathbb{K}$;
moreover, the exponent $2m/\left( m+1\right) $ is sharp$.$ Another natural
question is:

\bigskip

\textit{Is it possible to obtain multilinear versions of Littlewood's }$4/3$%
\textit{\ inequality keeping the exponent }$4/3?$

\bigskip

This problem was treated at least in two recent papers (we state the results
for complex scalars but the case of real scalars is similar, with slightly
different constants):

\begin{itemize}
\item (\cite{uni}) For all positive integers $N$ and all $m$-linear forms $%
U:\ell _{\infty }^{N}\times \cdots \times \ell _{\infty }^{N}\rightarrow
\mathbb{C}$ we have%
\begin{equation*}
\left( \sum\limits_{i,j=1}^{N}\left\vert
U(e_{i},...,e_{i},e_{j},...,e_{j})\right\vert ^{\frac{4}{3}}\right) ^{\frac{3%
}{4}}\leq \frac{2}{\sqrt{\pi }}\left\Vert U\right\Vert \text{.}
\end{equation*}

\item (\cite{arapell}) For all positive integers $N$ and all $m$-linear
forms $U:\ell _{\infty }^{N}\times \cdots \times \ell _{\infty
}^{N}\rightarrow \mathbb{C}$ we have
\begin{equation*}
\left( \sum\limits_{i_{1},\ldots ,i_{m}=1}^{N}\left\vert
U(e_{i_{^{1}}},\ldots ,e_{i_{m}})\right\vert ^{\frac{4}{3}}\right) ^{\frac{3%
}{4}}\leq \left( \prod\limits_{j=2}^{m}\Gamma \left( 2-\frac{1}{j}\right) ^{%
\frac{j}{2-2j}}\right) N^{\frac{m-2}{4}}\left\Vert U\right\Vert
\end{equation*}%
and the exponent $\frac{m-2}{4}$ is optimal.
\end{itemize}

In this paper we investigate this problem from a different point of view.
More precisely, as a consequence of our main result we show that for all
positive integers $m\geq 3$ and bijections $\sigma _{1},\ldots ,\sigma
_{m-2} $ from $\mathbb{N}\times \mathbb{N}$ to $\mathbb{N}$ we have%
\begin{equation*}
\left( \sum_{i,j=1}^{\infty }\left\vert U\left( e_{i},e_{j},e_{\sigma
_{1}(i,j)},\ldots ,e_{\sigma _{m-2}(i,j)}\right) \right\vert ^{\frac{4}{3}%
}\right) ^{\frac{3}{4}}\leq \sqrt{2}\Vert U\Vert
\end{equation*}%
for every continuous $m$-linear form $U:c_{0}\times \cdots \times
c_{0}\rightarrow \mathbb{K}$.

We prefer to begin with the theory of multiple summing operators and state
our main result in this context; then the above result (among others) will
be just simple consequences of the main result.

%

\section{Multiple summing operators}

Let $E,E_{1},...,E_{m}$ and $F$ denote Banach spaces over $\mathbb{K}$ and
let $B_{E^{\ast }}$ denote the closed unit ball of the topological dual of $%
E $. If $1\leq q\leq \infty $, by $q^{\ast }$ we represent the conjugate of $%
q$. For $p\geq 1$, by $\ell _{p}(E)$ we mean the space of absolutely $p$%
--summable sequences in $E$; also $\ell _{p}^{w}(E)$ denotes the linear
space of the sequences $\left( x_{j}\right) _{j=1}^{\infty }$ in $E$ such
that $\left( \varphi \left( x_{j}\right) \right) _{j=1}^{\infty }\in \ell
_{p}$ for every continuous linear functional $\varphi :E\rightarrow \mathbb{K%
}$. The function
\begin{equation*}
\left\Vert \left( x_{j}\right) _{j=1}^{\infty }\right\Vert
_{w,p}=\sup_{\varphi \in B_{E^{\ast }}}\left\Vert \left( \varphi \left(
x_{j}\right) \right) _{j=1}^{\infty }\right\Vert _{p}
\end{equation*}%
defines a norm on $\ell _{p}^{w}(E)$. The space of all continuous $m$-linear
operators $T:E_{1}\times \cdots \times E_{m}\rightarrow F$, with the $\sup $
norm, is denoted by $\mathcal{L}\left( E_{1},...,E_{m};F\right) $.

The notion of multiple summing operators, introduced independently by Matos
and P\'{e}rez-Garc\'{\i}a (\cite{matos, per}), is a natural extension of the
classical notion of absolutely summing linear operators (see \cite{Di}). But
multiple summing operators is certainly one of the most fruitful approaches
(see \cite{popa2, popa, popa3} for recent papers). For different approaches
we mention, for instance \cite{b, dimant, comparing,PeRuSa, RuSa}.

\begin{definition}
Let $1\leq q_{1},...,q_{m}\leq p<\infty $. A multilinear operator $T\in
\mathcal{L}\left( E_{1},...,E_{m};F\right) $ is multiple $\left(
p;q_{1},...,q_{m}\right) $--summing if there exists a $C>0$ such that
\begin{equation*}
\left( \sum_{j_{1},...,j_{m}=1}^{\infty }\left\Vert T\left(
x_{j_{1}}^{(1)},...,x_{j_{m}}^{(m)}\right) \right\Vert ^{p}\right) ^{\frac{1%
}{p}}\leq C\prod_{k=1}^{m}\left\Vert \left( x_{j_{k}}^{(k)}\right)
_{j_{k}=1}^{\infty }\right\Vert _{w,q_{k}}
\end{equation*}%
for all $\left( x_{j}^{(k)}\right) _{j=1}^{\infty }\in \ell
_{q_{k}}^{w}\left( E_{k}\right) $, $k\in \{1,...,m\}$. We represent the
class of all multiple $\left( p;q_{1},...,q_{m}\right) $--summing operators
from $E_{1},....,E_{m}$ to $F$ by $\Pi _{\mathrm{mult}\left(
p;q_{1},...,q_{m}\right) }\left( E_{1},...,E_{m};F\right) $ and $\pi _{%
\mathrm{mult}\left( p;q_{1},...,q_{m}\right) }\left( T\right) $ denotes the
infimum over all $C$ as above.
\end{definition}

The main result of this section is the following theorem. Its proof is
inspired in arguments from \cite{Ar,b}.

\begin{theorem}
\label{main} \label{iu}Let $n>m\geq 1$ be positive integers and $%
E_{1},...,E_{n},F$ Banach spaces. If%
\begin{equation}
\Pi _{\mathrm{mult}\left( p;q_{1},...,q_{m}\right) }\left(
E_{1},...,E_{m};F\right) =\mathcal{L}\left( E_{1},...,E_{m};F\right) ,
\label{hyp}
\end{equation}%
then there is a constant $C>0$ (not depending on $n$) such that
\begin{equation*}
\begin{array}{l}
\displaystyle\left( \sum_{i_{1},...,i_{m}=1}^{\infty }\left\Vert
U(x_{i_{1}}^{(1)},...,x_{i_{m}}^{(m)},x_{i_{1}\cdots
i_{m}}^{(m+1)},...,x_{i_{1}\cdots i_{m}}^{(n)})\right\Vert ^{p}\right) ^{%
\frac{1}{p}}\vspace{0.2cm} \\
\textstyle\qquad \leq C\Vert U\Vert \prod_{k=1}^{m}\left\Vert \left(
x_{i}^{(k)}\right) _{i=1}^{\infty }\right\Vert
_{w,q_{k}}\prod_{k=m+1}^{n}\left\Vert \left( x_{i_{1}\cdots
i_{m}}^{(k)}\right) _{i_{1},...,i_{m}=1}^{\infty }\right\Vert _{w,1},%
\end{array}%
\end{equation*}%
for all $n$-linear forms $U:E_{1}\times \cdots \times E_{n}\rightarrow F$.
\end{theorem}

\begin{proof}
The case $m=1$ is known (see \cite[Corollary 3.3]{b}). For $m\geq 2$ let us
proceed by induction on $n$. First we will show that the result holds for $%
n=m+1$. Let $N$ be a positive integer and $x_{i_{1}\cdots i_{m}}^{(m+1)}\in
E_{m+1}$. By the Hahn--Banach theorem we can choose norm one functionals $%
\varphi _{i_{1}\cdots i_{m}}$ such that
\begin{equation*}
\left\Vert U(x_{i_{1}}^{(1)},...,x_{i_{m}}^{(m)},x_{i_{1}\cdots
i_{m}}^{(m+1)})\right\Vert =\varphi _{i_{1}\cdots i_{m}}\left(
U(x_{i_{1}}^{(1)},...,x_{i_{m}}^{(m)},x_{i_{1}\cdots i_{m}}^{(m+1)})\right)
\end{equation*}%
for all $i_{1},...,i_{m}=1,\ldots ,N$.

A duality argument gives us non-negative real numbers $\alpha _{i_{1}\cdots
i_{m}}$ such that
\begin{equation*}
\sum_{i_{1},...,i_{m}=1}^{N}\alpha _{i_{1}\cdots i_{m}}^{p^{\ast }}=1,
\end{equation*}%
where $p^{\ast }$ is the conjugate number of $p$, i.e., $\frac{1}{p}+\frac{1%
}{p^{\ast }}=1$, and
\begin{equation*}
\begin{array}{l}
\displaystyle\left( \sum_{i_{1},...,i_{m}=1}^{N}\left\Vert
U(x_{i_{1}}^{(1)},...,x_{i_{m}}^{(m)},x_{i_{1}\cdots
i_{m}}^{(m+1)})\right\Vert ^{p}\right) ^{\frac{1}{p}} \\
\displaystyle=\sum_{i_{1},...,i_{m}=1}^{N}\alpha _{i_{1}\cdots
i_{m}}\left\Vert U(x_{i_{1}}^{(1)},...,x_{i_{m}}^{(m)},x_{i_{1}\cdots
i_{m}}^{(m+1)})\right\Vert \\
\displaystyle=\sum_{i_{1},...,i_{m}=1}^{N}\alpha _{i_{1}\cdots i_{m}}\varphi
_{i_{1}\cdots i_{m}}\left(
U(x_{i_{1}}^{(1)},...,x_{i_{m}}^{(m)},x_{i_{1}\cdots i_{m}}^{(m+1)})\right) .%
\end{array}%
\end{equation*}

Let $r_{j_{1}\cdots j_{m}}$ be the Rademacher functions indexed on $\mathbb{N%
}\times \cdots \times \mathbb{N}$ (the order is not important). We have
\begin{align*}
& \int_{0}^{1}\sum_{i_{1},...,i_{m}=1}^{N}r_{i_{1}\cdots i_{m}}(t)\alpha
_{i_{1}\cdots i_{m}}\varphi _{i_{1}\cdots i_{m}}\left(
U(x_{i_{1}}^{(1)},...,x_{i_{m}}^{(m)},\sum_{j_{1},...,j_{m}=1}^{N}r_{j_{1}%
\cdots j_{m}}(t)x_{j_{1}\cdots j_{m}}^{(m+1)})\right) \,dt \\
& =\sum_{i_{1},...,i_{m}=1}^{N}\sum_{j_{1},...,j_{m}=1}^{N}\alpha
_{i_{1}\cdots i_{m}}\varphi _{i_{1}\cdots i_{m}}\left(
U(x_{i_{1}}^{(1)},...,x_{i_{m}}^{(m)},x_{j_{1}\cdots j_{m}}^{(m+1)})\right)
\int_{0}^{1}r_{i_{1}\cdots i_{m}}(t)r_{j_{1}\cdots j_{m}}(t)\,dt \\
& =\sum_{i_{1},...,i_{m}=1}^{N}\alpha _{i_{1}\cdots i_{m}}\varphi
_{i_{1}\cdots i_{m}}\left(
U(x_{i_{1}}^{(1)},...,x_{i_{m}}^{(m)},x_{i_{1}\cdots i_{m}}^{(m+1)})\right)
\\
& =\left( \sum_{i_{1},...,i_{m}=1}^{N}\left\Vert
U(x_{i_{1}}^{(1)},...,x_{i_{m}}^{(m)},x_{i_{1}\cdots
i_{m}}^{(m+1)})\right\Vert ^{p}\right) ^{\frac{1}{p}}.
\end{align*}%
Hence%
\begin{align*}
& \left( \sum_{i_{1},...,i_{m}=1}^{N}\left\Vert
U(x_{i_{1}}^{(1)},...,x_{i_{m}}^{(m)},x_{i_{1}\cdots
i_{m}}^{(m+1)})\right\Vert ^{p}\right) ^{\frac{1}{p}} \\
& \leq \int_{0}^{1}\left\vert \sum_{i_{1},...,i_{m}=1}^{N}r_{i_{1}\cdots
i_{m}}(t)\alpha _{i_{1}\cdots i_{m}}\varphi _{i_{1}\cdots i_{m}}\Big(%
U(x_{i_{1}}^{(1)},...,x_{i_{m}}^{(m)},\sum_{j_{1},...,j_{m}=1}^{N}r_{j_{1}%
\cdots j_{m}}(t)x_{j_{1}\cdots j_{m}}^{(m+1)})\Big)\right\vert \,\,dt \\
& \leq \sup_{t\in \lbrack 0,1]}\left\vert
\sum_{i_{1},...,i_{m}=1}^{N}r_{i_{1}\cdots i_{m}}(t)\alpha _{i_{1}\cdots
i_{m}}\varphi _{i_{1}\cdots i_{m}}\Big(%
U(x_{i_{1}}^{(1)},...,x_{i_{m}}^{(m)},\sum_{j_{1},...,j_{m}=1}^{N}r_{j_{1}%
\cdots j_{m}}(t)x_{j_{1}\cdots j_{m}}^{(m+1)})\Big)\right\vert \\
& \leq \sup_{t\in \lbrack 0,1]}\sum_{i_{1},...,i_{m}=1}^{N}\alpha
_{i_{1}\cdots i_{m}}\Big\Vert U\Big(x_{i_{1}}^{(1)},...,x_{i_{m}}^{(m)},%
\sum_{j_{1},...,j_{m}=1}^{N}r_{j_{1}\cdots j_{m}}(t)x_{j_{1}\cdots
j_{m}}^{(m+1)}\Big)\Big\Vert \\
& \leq \left( \sum_{i_{1},...,i_{m}=1}^{N}\alpha _{i_{1}\cdots
i_{m}}^{p^{\ast }}\right) ^{\frac{1}{p^{\ast }}}\cdot \sup_{t\in \lbrack
0,1]}\left( \sum_{i_{1},...,i_{m}=1}^{N}\Big\Vert %
U(x_{i_{1}}^{(1)},...,x_{i_{m}}^{(m)},\sum_{j_{1},...,j_{m}=1}^{N}r_{j_{1}%
\cdots j_{m}}(t)x_{j_{1}\cdots j_{m}}^{(m+1)})\Big\Vert^{p}\right) ^{\frac{1%
}{p}} \\
& \leq \sup_{t\in \lbrack 0,1]}\pi _{(p;q_{1},...,q_{m})}\Big(U(\cdot
,...,\cdot ,\sum_{j_{1},...,j_{m}=1}^{N}r_{j_{1}\cdots
j_{m}}(t)x_{j_{1}\cdots j_{m}}^{(m+1)})\Big)\prod_{k=1}^{m}\left\Vert \left(
x_{i}^{(k)}\right) _{i=1}^{N}\right\Vert _{w,q_{k}}
\end{align*}%
where in the last inequality we have used \eqref{hyp}. From \eqref{hyp} it
follows from the Open Mapping Theorem that there is a constant $C>0$ such
that $\pi _{(p;q_{1},...,q_{m})}(\ \cdot \ )\leq C\Vert \cdot \Vert $. Then
\begin{equation*}
\begin{array}{l}
\displaystyle\left( \sum_{i_{1},...,i_{m}=1}^{N}\left\Vert
U(x_{i_{1}}^{(1)},...,x_{i_{m}}^{(m)},x_{i_{1}\cdots
i_{m}}^{(m+1)})\right\Vert ^{p}\right) ^{\frac{1}{p}} \\
\displaystyle\leq \sup_{t\in \lbrack 0,1]}C\left\Vert U(\cdot ,...,\cdot
,\sum_{j_{1},...,j_{m}=1}^{N}r_{j_{1}\cdots j_{m}}(t)x_{j_{1}\cdots
j_{m}}^{(m+1)})\right\Vert \prod_{k=1}^{m}\left\Vert \left(
x_{i}^{(k)}\right) _{i=1}^{N}\right\Vert _{w,q_{k}} \\
\displaystyle\leq C\Vert U\Vert \sup_{t\in \lbrack 0,1]}\left\Vert
\sum_{i_{1},...,i_{m}=1}^{N}r_{i_{1}\cdots i_{m}}(t)x_{i_{1}\cdots
i_{m}}^{(m+1)}\right\Vert \prod_{k=1}^{m}\left\Vert \left(
x_{i}^{(k)}\right) _{i=1}^{N}\right\Vert _{w,q_{k}} \\
\displaystyle\leq C\Vert U\Vert \left( \prod_{k=1}^{m}\left\Vert \left(
x_{i}^{(k)}\right) _{i=1}^{N}\right\Vert _{w,q_{k}}\right) \left\Vert \left(
x_{i_{1}\cdots i_{m}}^{(m+1)}\right) _{i_{1},...,i_{m}=1}^{N}\right\Vert
_{w,1}.%
\end{array}%
\end{equation*}

The proof is completed by an induction argument, as follows. Suppose that
the result is valid for a positive integer $n\geq m+1$. Let $N$ be a
positive integer and $E_{n+1}$ a Banach space. Let $x_{i_{1}\cdots
i_{m}}^{(n+1)}\in E_{n+1}$ and norm one functionals $\varphi _{i_{1}\cdots
i_{m}}$ such that%
\begin{equation*}
\begin{array}{l}
\displaystyle\left\Vert U(x_{i_{1}}^{(1)},...,x_{i_{m}}^{(m)},x_{i_{1}\cdots
i_{m}}^{(m+1)},\ldots ,x_{i_{1}\cdots i_{m}}^{(n+1)})\right\Vert \\
\displaystyle=\varphi _{i_{1}\cdots i_{m}}\left(
U(x_{i_{1}}^{(1)},...,x_{i_{m}}^{(m)},x_{i_{1}\cdots i_{m}}^{(m+1)},\ldots
,x_{i_{1}\cdots i_{m}}^{(n+1)})\right) ,%
\end{array}%
\end{equation*}%
for all $i_{1},...,i_{m}=1,\ldots ,N$. A duality argument gives us
non-negative real numbers $\alpha _{i_{1}\cdots i_{m}}$ such that
\begin{equation*}
\sum_{i_{1},...,i_{m}=1}^{N}\alpha _{i_{1}\cdots i_{m}}^{p^{\ast }}=1
\end{equation*}%
and%
\begin{align*}
& \left( \sum_{i_{1},...,i_{m}=1}^{N}\left\Vert
U(x_{i_{1}}^{(1)},...,x_{i_{m}}^{(m)},x_{i_{1}\cdots i_{m}}^{(m+1)},\ldots
,x_{i_{1}\cdots i_{m}}^{(n+1)})\right\Vert ^{p}\right) ^{\frac{1}{p}} \\
& =\sum_{i_{1},...,i_{m}=1}^{N}\alpha _{i_{1}\cdots i_{m}}\left\Vert
U(x_{i_{1}}^{(1)},...,x_{i_{m}}^{(m)},x_{i_{1}\cdots i_{m}}^{(m+1)},\ldots
,x_{i_{1}\cdots i_{m}}^{(n+1)})\right\Vert \\
& =\sum_{i_{1},...,i_{m}=1}^{N}\alpha _{i_{1}\cdots i_{m}}\varphi
_{i_{1}\cdots i_{m}}\left(
U(x_{i_{1}}^{(1)},...,x_{i_{m}}^{(m)},x_{i_{1}\cdots i_{m}}^{(m+1)},\ldots
,x_{i_{1}\cdots i_{m}}^{(n+1)})\right) .
\end{align*}

We also have
\begin{align*}
& \int_{0}^{1}\sum_{i_{1},...,i_{m}=1}^{N}r_{i_{1}\cdots i_{m}}(t)\alpha
_{i_{1}\cdots i_{m}} \\
& \quad \times \varphi _{i_{1}\cdots i_{m}}\left(
U(x_{i_{1}}^{(1)},...,x_{i_{m}}^{(m)},x_{i_{1}\cdots i_{m}}^{(m+1)},\ldots
,x_{i_{1}\cdots i_{m}}^{(n)},\sum_{j_{1},...,j_{m}=1}^{N}r_{j_{1}\cdots
j_{m}}(t)x_{j_{1}\cdots j_{m}}^{(n+1)})\right) \,dt \\
& =\sum_{i_{1},...,i_{m}=1}^{N}\sum_{j_{1},...,j_{m}=1}^{N}\alpha
_{i_{1}\cdots i_{m}}\varphi _{i_{1}\cdots i_{m}}\left(
U(x_{i_{1}}^{(1)},...,x_{i_{m}}^{(m)},x_{i_{1}\cdots i_{m}}^{(m+1)},\ldots
,x_{i_{1}\cdots i_{m}}^{(n)},x_{j_{1}\cdots j_{m}}^{(n+1)})\right) \\
& \quad \times \int_{0}^{1}r_{i_{1}\cdots i_{m}}(t)r_{j_{1}\cdots
j_{m}}(t)\,dt \\
& =\sum_{i_{1},...,i_{m}=1}^{N}\alpha _{i_{1}\cdots i_{m}}\varphi
_{i_{1}\cdots i_{m}}\left(
U(x_{i_{1}}^{(1)},...,x_{i_{m}}^{(m)},x_{i_{1}\cdots i_{m}}^{(m+1)},\ldots
,x_{i_{1}\cdots i_{m}}^{(n+1)})\right) \\
& =\left( \sum_{i_{1},...,i_{m}=1}^{N}\left\Vert
U(x_{i_{1}}^{(1)},...,x_{i_{m}}^{(m)},x_{i_{1}\cdots i_{m}}^{(m+1)},\ldots
,x_{i_{1}\cdots i_{m}}^{(n+1)})\right\Vert ^{p}\right) ^{\frac{1}{p}}.
\end{align*}%
Hence using the induction hypothesis
\begin{align*}
& \left( \sum_{i_{1},...,i_{m}=1}^{N}\left\Vert
U(x_{i_{1}}^{(1)},...,x_{i_{m}}^{(m)},x_{i_{1}\cdots i_{m}}^{(m+1)},\ldots
,x_{i_{1}\cdots i_{m}}^{(n+1)})\right\Vert ^{p}\right) ^{\frac{1}{p}} \\
& \leq \int_{0}^{1}\left\vert \sum_{i_{1},...,i_{m}=1}^{N}r_{i_{1}\cdots
i_{m}}(t)\alpha _{i_{1}\cdots i_{m}}\right. \\
& \quad \left. \times \varphi _{i_{1}\cdots i_{m}}\left(
U(x_{i_{1}}^{(1)},...,x_{i_{m}}^{(m)},x_{i_{1}\cdots i_{m}}^{(m+1)},\ldots
,x_{i_{1}\cdots i_{m}}^{(n)},\sum_{j_{1},...,j_{m}=1}^{N}r_{j_{1}\cdots
j_{m}}(t)x_{j_{1}\cdots j_{m}}^{(n+1)})\right) \right\vert \,dt \\
& \leq \sup_{t\in \lbrack 0,1]}\sum_{i_{1},...,i_{m}=1}^{N}\alpha
_{i_{1}\cdots i_{m}}\Big\Vert U(x_{i_{1}}^{(1)},...,x_{i_{m}}^{(m)},x_{i_{1}%
\cdots i_{m}}^{(m+1)},\ldots ,x_{i_{1}\cdots
i_{m}}^{(n)},\sum_{j_{1},...,j_{m}=1}^{N}r_{j_{1}\cdots
j_{m}}(t)x_{j_{1}\cdots j_{m}}^{(n+1)})\Big\Vert \\
& \leq \left( \sum_{i_{1},...,i_{m}=1}^{N}\alpha _{i_{1}\cdots
i_{m}}^{p^{\ast }}\right) ^{\frac{1}{p^{\ast }}}\cdot \sup_{t\in \lbrack
0,1]}\left( \sum_{i_{1},...,i_{m}=1}^{N}\Big\Vert %
U(x_{i_{1}}^{(1)},...,x_{i_{m}}^{(m)},x_{i_{1}\cdots i_{m}}^{(m+1)},\ldots
,x_{i_{1}\cdots i_{m}}^{(n)},\right. \\
& \qquad \left. \sum_{j_{1},...,j_{m}=1}^{N}r_{j_{1}\cdots
j_{m}}(t)x_{j_{1}\cdots j_{m}}^{(n+1)})\Big\Vert^{p}\right) ^{\frac{1}{p}} \\
& \leq \sup_{t\in \lbrack 0,1]}C\Big\Vert U\Big(\cdot ,...,\cdot
,\sum_{j_{1},...,j_{m}=1}^{N}r_{j_{1}\cdots j_{m}}(t)x_{j_{1}\cdots
j_{m}}^{(n+1)}\Big)\Big\Vert \\
& \qquad\times \left( \prod_{k=1}^{m}\Big\Vert\Big(x_{i}^{(k)}\Big)_{i=1}^{N}%
\Big\Vert_{w,q_{k}}\right) \left( \prod_{k=m+1}^{n}\Big\Vert\Big(%
x_{i_{1}\cdots i_{m}}^{(k)}\Big)_{i_{1},...,i_{m}=1}^{N}\Big\Vert%
_{w,1}\right) \\
& \leq C\left\Vert U\right\Vert \left( \prod_{k=1}^{m}\left\Vert \left(
x_{i}^{(k)}\right) _{i=1}^{N}\right\Vert _{w,q_{k}}\right) \left(
\prod_{k=m+1}^{n+1}\left\Vert \left( x_{i_{1}\cdots i_{m}}^{(k)}\right)
_{i_{1},...,i_{m}=1}^{N}\right\Vert _{w,1}\right) .
\end{align*}%
Now we just make $N\rightarrow \infty .$
\end{proof}

\begin{example}
If $F$ is a Banach space with cotype $2$ it is well known that $\Pi _{%
\mathrm{mult}\left( 2;2,...,2\right) }\left( ^{m}c_{0};F\right) =\mathcal{L}%
\left( ^{m}c_{0};F\right) .$ From the above theorem we conclude that%
\begin{equation*}
\begin{array}{l}
\displaystyle\left( \sum_{i_{1},...,i_{m}=1}^{\infty }\left\Vert
U(x_{i_{1}}^{(1)},...,x_{i_{m}}^{(m)},x_{i_{1}\cdots
i_{m}}^{(m+1)},...,x_{i_{1}\cdots i_{m}}^{(n)})\right\Vert ^{2}\right) ^{%
\frac{1}{2}}\vspace{0.2cm} \\
\textstyle\qquad \leq C\Vert U\Vert \prod_{k=1}^{m}\left\Vert \left(
x_{i}^{(k)}\right) _{i=1}^{\infty }\right\Vert
_{w,2}\prod_{k=m+1}^{n}\left\Vert \left( x_{i_{1}\cdots i_{m}}^{(k)}\right)
_{i_{1},...,i_{m}=1}^{\infty }\right\Vert _{w,1},%
\end{array}%
\end{equation*}%
regardless of the Banach space $E$ and regardless of the $n$-linear operator
$U:c_{0}\times \overset{m\text{ times}}{\cdots }\times c_{0}\times E\times
\cdots \times E\rightarrow F$.
\end{example}

\begin{remark}
The constant $C$ that appears in the above theorem can be chosen as the
constant from the Open Mapping Theorem used in the coincidence (\ref{hyp}).
\end{remark}

\begin{remark}
The case $F=\mathbb{K}$ and $m=1$ with $q=p=1$ recovers the Defant-Voigt
Theorem (see \cite{alencar}).
\end{remark}

\begin{remark}
Theorem \ref{iu} is in some sense optimal. In fact it was recently proved in
\cite{aaa} that the Defant-Voigt Theorem is optimal in the following sense:
every continuous $m$-linear form is absolutely $(1;1,...,1)$-summing and
this result can not be improved to $(p;1,...,1)$-summing with $p<1$.
\end{remark}

\section{Some applications}

In this section we show how the result proved in the previous section is
connected to the problem stated in the introduction of this note.

\subsection{Variations of Littlewood's $4/3$ theorem and Bohnenblust--Hille
inequality}

We begin by proving the result stated in the Introduction:

\begin{theorem}
\label{89}Let $m\geq 3$ be an integer and $\sigma _{1},\ldots ,\sigma _{m-2}$
be bijections from $\mathbb{N}\times \mathbb{N}$ to $\mathbb{N}$. Then
\begin{equation}
\left( \sum_{i,j=1}^{\infty }\left\vert U\left( e_{i},e_{j},e_{\sigma
_{1}(i,j)},\ldots ,e_{\sigma _{n-2}(i,j)}\right) \right\vert ^{\frac{4}{3}%
}\right) ^{\frac{3}{4}}\leq \sqrt{2}\Vert U\Vert  \label{estr}
\end{equation}%
for all continuous $m$-linear forms $U:c_{0}\times \cdots \times
c_{0}\rightarrow \mathbb{K}$. 
\end{theorem}

\begin{proof}
From Littlewood's $4/3$ theorem we know that $\Pi _{\mathrm{mult}\left(
4/3;1,1\right) }\left( ^{2}c_{0};\mathbb{K}\right) =\mathcal{L}\left(
^{2}c_{0};\mathbb{K}\right) $ and the constant involved is $\sqrt{2}$ (or $2/%
\sqrt{\pi }$ for complex scalars)$.$ By choosing $x_{ij}^{(k)}=e_{\sigma
_{k}(i,j)},$ since $\left\Vert \left( e_{\sigma _{k}(i,j)}\right)
_{ij=1}^{\infty }\right\Vert _{w,1}=1$ the proof is done.
\end{proof}

The same argument of the previous theorem can be used to prove the following
more general result:

\begin{theorem}
Let $n>m\geq 1$ be positive integers and $\sigma _{k}:\mathbb{N}^m \to \mathbb{N}$ be bijections for all $%
k=1,...,n-m$. Then there is a constant $L_{m}^{\mathbb{K}}\geq 1$ such that%
\begin{equation*}
\left( \sum\limits_{i_{1},...,i_{m}=1}^{\infty }\left\vert
U(e_{i_{1}},...,e_{i_{m}},e_{\sigma _{1}(i_{1},...,i_{m})},\ldots ,e_{\sigma
_{n-m}(i_{1},...,i_{m})})\right\vert ^{\frac{2m}{m+1}}\right) ^{\frac{m+1}{2m%
}}\leq L_{m}^{\mathbb{K}}\left\Vert U\right\Vert
\end{equation*}%
for every bounded $n$-linear form $U:c_{0}\times \cdots \times
c_{0}\rightarrow \mathbb{K}$. 
\end{theorem}

\begin{remark}
As a matter of fact, the constants $L_{m}^{\mathbb{K}}$ can be estimated.
From the proof of Theorem \ref{iu} it is simple to see that $L_{m}^{\mathbb{K%
}}$ can be chosen as the best known constants of the Bohnenblust--Hille
inequality. So, using the estimates of \cite{bohr}, we know that
\begin{equation}
\begin{array}{llll}
L_{m}^{\mathbb{C}} & \leq & \displaystyle\prod\limits_{j=2}^{m}\Gamma \left(
2-\frac{1}{j}\right) ^{\frac{j}{2-2j}}, & \vspace{0.2cm} \\
L_{m}^{\mathbb{R}} & \leq & 2^{\frac{446381}{55440}-\frac{k}{2}}\displaystyle%
\prod\limits_{j=14}^{k}\left( \frac{\Gamma \left( \frac{3}{2}-\frac{1}{j}%
\right) }{\sqrt{\pi }}\right) ^{\frac{j}{2-2j}}, & \text{ for }m\geq 14,%
\vspace{0.2cm} \\
L_{m}^{\mathbb{R}} & \leq & \displaystyle\left( \sqrt{2}\right)
^{\sum_{j=1}^{k-1}\frac{1}{j}}, & \text{ for }2\leq m\leq 13.%
\end{array}
\label{bkecbh}
\end{equation}
\end{remark}

The above estimates can be rewritten as (see \cite{bohr})
\begin{equation*}
\begin{array}{rcl}
L_{m}^{\mathbb{C}} & < & \displaystyle m^{0.21139},\vspace{0.2cm} \\
L_{m}^{\mathbb{R}} & < & 1.3\times m^{0.36482},%
\end{array}%
\end{equation*}
The extension of the Bohnenblust--Hille inequality to $\ell _{p}$ spaces in
the place of $\ell _{\infty }$ spaces is divided in two cases: $m<p\leq 2m$
and $p\geq 2m$. The case $p\geq 2m,$ sometimes called
Hardy--Littlewood/Praciano-Pereira inequality (see \cite{hardy,pra}) states
that there exists a (optimal) constant $C_{m,p}^{\mathbb{K}}\geq 1$ such
that, for all positive integers $N $ and all $m$-linear forms $T:\ell
_{p}^{N}\times \cdots \times \ell _{p}^{N}\rightarrow \mathbb{K}$,
\begin{equation}
\left( \sum_{i_{1},\ldots ,i_{m}=1}^{N}\left\vert T(e_{i_{1}},\ldots
,e_{i_{m}})\right\vert ^{\frac{2mp}{mp+p-2m}}\right) ^{\frac{mp+p-2m}{2mp}%
}\leq C_{m,p}^{\mathbb{K}}\left\Vert T\right\Vert .  \label{i99}
\end{equation}

\section{\protect\bigskip Final remark}

When $m<p\leq 2m$ the Hardy--Littlewood inequality is also known as
Hardy--Littlewood/Dimant-Sevilla--Peris inequality (\cite{dimant2,hardy}).
It reads as follows:

\begin{theorem}[Hardy--Littlewood/Dimant--Sevilla-Peris]
\label{yytt}For $m<p\leq 2m$, there is a constant $C_{\mathbb{K},m,p}\geq 1$
such that%
\begin{equation*}
\left( \sum_{i_{1},...,i_{m}=1}^{N}\left\vert T(e_{i_{1}},\ldots
,e_{i_{m}})\right\vert ^{\frac{p}{p-m}}\right) ^{\frac{p-m}{p}}\leq C_{%
\mathbb{K},m,p}\left\Vert T\right\Vert
\end{equation*}%
for all positive integers $N$ and all $m$-linear form $T:\ell _{p}^{N}\times
\cdots \times \ell _{p}^{N}\rightarrow \mathbb{K}$. Moreover the exponent $%
\frac{p}{p-m}$ is optimal.
\end{theorem}

In this case we can prove the following (optimal) result, which does not
depend on the results developed in the previous sections:

\begin{proposition}
\label{90}Let $m>n\geq 1$ be positive integers, let $m<p\leq 2m$ and $\sigma
_{k}:\mathbb{N}^n \to \mathbb{N}$ be
bijections for all $k=1,....,n$. Then there is a constant $C_{\mathbb{K}%
,m,n,p}\geq 1$ such that%
\begin{equation*}
\left( \sum\limits_{i_{1},...,i_{n}=1}^{N}\left\vert
U(e_{i_{1}},...,e_{i_{n}},e_{\sigma _{1}(i_{1},...,i_{n})},\ldots ,e_{\sigma
_{m-n}(i_{1},...,i_{n})})\right\vert ^{\frac{p}{p-m}}\right) ^{\frac{p-m}{p}%
}\leq C_{\mathbb{K},m,n,p}\left\Vert U\right\Vert
\end{equation*}%
for all positive integers $N\geq 1$ and all continuous $m$-linear form $%
U:\ell _{p}^{N}\times \cdots \times \ell _{p}^{N}\rightarrow \mathbb{K}$.
Moreover, the exponent $p/(p-m)$ is optimal.
\end{proposition}

\begin{proof}
For the sake of simplicity we suppose $n=2$. The general case is similar.
Note that, using Theorem \ref{yytt} we have
\begin{align*}
& \left( \sum_{i,j=1}^{N}\left\vert U(e_{i},e_{j},e_{\sigma
_{1}(i,j)},\ldots ,e_{\sigma _{m-2}(i,j)})\right\vert ^{\frac{p}{p-m}%
}\right) ^{\frac{p-m}{p}} \\
& \leq \left( \sum_{i_{1},...,i_{m}=1}^{\infty }\left\vert
U(e_{i_{1}},\ldots ,e_{i_{m}})\right\vert ^{\frac{p}{p-m}}\right) ^{\frac{p-m%
}{p}}\leq C_{\mathbb{K},m,p}\left\Vert U\right\Vert .
\end{align*}%
The optimality of the exponent $\frac{p}{p-m}$ is proved next using the same
argument of the proof of the theorem of
Hardy--Littlewood/Dimant-Sevilla-Peris (see \cite{dimant2,hardy}). Consider
the $m$-linear form
\begin{equation*}
U:\ell _{p}\times \cdots \times \ell _{p}\rightarrow \mathbb{K}
\end{equation*}%
given by
\begin{equation*}
U(x^{(1)},...,x^{(m)})=\textstyle\sum%
\limits_{i=1}^{N}x_{i}^{(1)}x_{i}^{(2)}x_{\sigma _{1}(i,i)}^{(3)}\cdots
x_{\sigma _{m-2}(i,i)}^{(m)}.
\end{equation*}%
From H\"{o}lder's inequality we have
\begin{equation*}
\left\Vert U\right\Vert \leq N^{\frac{p-m}{p}}.
\end{equation*}%
If the theorem is valid for a power $s$, then%
\begin{align*}
& \left( \sum_{i=1}^{N}\left\vert U\left( e_{i},e_{i},e_{\sigma
_{1}(i,i)},\ldots ,e_{\sigma _{m-2}(i,i)}\right) \right\vert ^{s}\right) ^{%
\frac{1}{s}} \\
& =\left( \sum_{i,j=1}^{N}\left\vert U\left( e_{i},e_{j},e_{\sigma
_{1}(i,j)},\ldots ,e_{\sigma _{m-2}(i,j)}\right) \right\vert ^{s}\right) ^{%
\frac{1}{s}} \\
& \leq C_{\mathbb{K},m,p}\left\Vert U\right\Vert \leq C_{\mathbb{K},m,p}N^{%
\frac{p-m}{p}}
\end{align*}%
and thus%
\begin{equation*}
N^{\frac{1}{s}}\leq C_{\mathbb{K},m,p}N^{\frac{p-m}{p}}
\end{equation*}%
and hence%
\begin{equation*}
s\geq \frac{p}{p-m}.
\end{equation*}
\end{proof}

It is important to recall that a somewhat similar inequality due to
Zalduendo asserts that%
\begin{equation*}
\left( \sum_{i=1}^{n}\left\vert T(e_{i},\ldots ,e_{i})\right\vert ^{\frac{p}{%
p-m}}\right) ^{\frac{p-m}{p}}\leq \left\Vert T\right\Vert
\end{equation*}%
for all positive integers $n$ and all $m$-linear forms $T:\ell
_{p}^{n}\times \cdots \times \ell _{p}^{n}\rightarrow \mathbb{K}$ and the
exponent $\frac{p}{p-m}$ is optimal. Note that Zalduendo's result and
Proposition \ref{90} are slightly different.

\bigskip

\section{\protect\bigskip Appendix: a variation of the
Kahane--Salem--Zygmund and applications}

In this section we follow a method of \cite{BAYART} to prove this.
Let us denote by $\psi _{2}(x):=\exp (x^{2})-1$ for $x\geq 0$. Let $(\Omega ,%
\mathcal{A},\mathbb{P})$ be a probability measure space and let us consider
the Orlicz space $L_{\psi _{2}}=L_{\psi _{2}}(\Omega ,\mathcal{A},\mathbb{P})
$ associated to $\psi _{2}$ formed by all real-valued random variables $X$
on $(\Omega ,\mathcal{A},\mathbb{P})$ such that $\mathbb{E}\left( \psi
_{2}\left( |X|/c\right) \right) <\infty $ for some $c>0$. The associated
Orlicz norm $\Vert \cdot \Vert _{\psi _{2}}$ is given by
\begin{equation*}
\Vert X\Vert _{\psi _{2}}:=\inf \{c>0\,;\,\mathbb{E}\left( \psi _{2}\left(
|X|/c\right) \right) \leq 1\},
\end{equation*}%
and $\left( L_{\psi _{2}},\Vert \cdot \Vert _{\psi _{2}}\right) $ is a
Banach space. We shall use the following lemma, which was suggested to us by
F. Bayart.

\begin{lemma}
Let $M$ be a metric space and let $(X(\omega,x))$ a family of random
variables defined on $(\Omega,\mathcal{A},\mathbb{P})$ and indexed by $M$.
Assume that there exists $A>0$ and a finite set $F\subset M$ such that

\begin{itemize}
\item[i.] For any $x\in M$, $\Vert X(\cdot ,x)\Vert _{\psi _{2}}\leq A$;

\item[ii.] For any $x\in M$, there exists $y\in F$ such that
\begin{equation*}
\sup_{\omega\in\Omega} |X(\omega,x)-X(\omega ,y)| \leq \frac{1}{2}
\sup_{z\in M}|X(\omega,z)|.
\end{equation*}
\end{itemize}

Then for any $R>0$ with $\frac{\text{card}(F)}{\psi_2(R/A)}<1$, there exists
$\omega\in \Omega$ satisfying
\begin{equation*}
\sup_{x\in M}|X(\omega,x)|\leq 2R.
\end{equation*}
\end{lemma}

\begin{proof}
This is exactly what is done in \cite{BAYART}, Step 2 and Step 3 of the
proof of Theorem 3.1, in an abstract context. For the sake of completeness,
we give the details. Given $x\in M$, condition $(ii)$ provides us $y\in F$
such that
\begin{equation*}
\sup_{\omega \in \Omega }|X(\omega ,x)-X(\omega ,y)|\leq \frac{1}{2}%
\sup_{z\in M}|X(\omega ,z)|.
\end{equation*}%
From
\begin{equation*}
\left\vert X(\omega ,x)\right\vert \leq \left\vert X(\omega ,x)-X(\omega
,y)\right\vert +\left\vert X(\omega ,y)\right\vert \leq \frac{1}{2}%
\sup_{z\in M}|X(\omega ,z)|+\sup_{w\in F}|X(\omega ,w)|
\end{equation*}%
we get that, for any $\omega \in \Omega $,
\begin{equation}
\sup_{x\in M}|X(\omega ,x)|\leq 2\sup_{w\in F}|X(\omega ,w)|.
\label{prob_lemma}
\end{equation}%
Let us fix $R>0$. As in the Step (3) of \cite[Theorem 3.1]{BAYART} we have
\begin{equation*}
\mathbb{P}\left( \left\{ \omega \in \Omega \,;\,|X(\omega ,x)|>R\right\}
\right) =\mathbb{P}\left( \left\{ \omega \in \Omega \,;\,\psi _{2}\left(
\frac{|X(\omega ,x)|}{A}\right) >\psi _{2}\left( \frac{R}{A}\right) \right\}
\right)
\end{equation*}%
The Markov inequality leads us to
\begin{equation*}
\mathbb{P}\left( \left\{ \omega \in \Omega \,;\,|X(\omega ,x)|>R\right\}
\right) \leq \frac{\mathbb{E}\left( \psi _{2}\left( |X(\omega ,x)|/A\right)
\right) }{\psi _{2}(R/A)}.
\end{equation*}%
Condition $(i)$ provides $\Vert X(\cdot ,x)\Vert _{\psi _{2}}\leq A$, thus
the definition of $\Vert \cdot \Vert _{\psi _{2}}$ assures that $\mathbb{E}%
\left( \psi _{2}\left( |X(\omega ,x)|/A\right) \right) \leq 1$.
Consequently, we get that for any $\omega \in \Omega $,
\begin{equation*}
\mathbb{P}\left( \left\{ \omega \in \Omega \,;\,|X(\omega ,x)|>R\right\}
\right) \leq \frac{1}{\psi _{2}(R/A)}.
\end{equation*}%
Since $F\subset M$ is finite,
\begin{equation*}
\mathbb{P}\left( \left\{ \omega \in \Omega \,;\,\sup_{w\in F}|X(\omega
,w)|>R\right\} \right) \leq \frac{card\,F}{\psi _{2}(R/A)}.
\end{equation*}%
Combining this with \eqref{prob_lemma} we get
\begin{equation*}
\mathbb{P}\left( \left\{ \omega \in \Omega \,;\,\sup_{x\in M}|X(\omega
,x)|>2R\right\} \right) \leq \frac{card\,F}{\psi _{2}(R/A)}.
\end{equation*}%
Thus, if we take $R>0$ such that $\displaystyle\frac{\text{card}(F)}{\psi
_{2}(R/A)}<1$, then
\begin{equation*}
\mathbb{P}\left( \left\{ \omega \in \Omega \,;\,\sup_{x\in M}|X(\omega
,x)|\leq 2R\right\} \right) >0.
\end{equation*}%
Therefore, there exists $\omega \in \Omega $ satisfying
\begin{equation*}
\sup_{x\in M}|X(\omega ,x)|\leq 2R.
\end{equation*}
\end{proof}

The previous approach can be applied in the following situation: let $N\geq 1$ and let $(\varepsilon _{i})_{i\in \{1,\dots ,N\}^{k}}$ be a sequence of independent Bernoulli variables defined on the same probability space $(\Omega ,\mathcal{A}, \mathbb{P})$. Let $M$ be the unit ball of $(\ell _{\infty }^{N})^{n}$ (endowed with the sup norm). For $x=(x^{(1)},\dots ,x^{(n)})$ in $M$ we define for positive integers $n_{1}+\dots +n_{k}=n$ and $j_{l}=n_{1}+\cdots +n_{l},\,l=1,\dots ,k$ 

\begin{equation*}
X(\omega ,x)=\sum_{i\in \{1,\dots ,N\}^{k}}\varepsilon _{i}(\omega
)x_{i_{1}}^{(1)}\cdots x_{i_{1}}^{(j_{1})}x_{i_{2}}^{(j_{1}+1)}\cdots
x_{i_{2}}^{(j_{2})}\cdots x_{i_{k}}^{(j_{k-1}+1)}\cdots x_{i_{k}}^{(j_{k})}
\end{equation*}%
For a fixed value of $x$, the $L^{2}$-norm of this random process can be
majorized, using the Khinchin inequality:
\begin{align*}
\Vert X(\cdot ,x)\Vert _{2}& =\left( \int\limits_{\Omega }\left\vert
X(w,x)\right\vert ^{2}d\mathbb{P}\right) ^{1/2} \\
& =\left( \int\limits_{\Omega }\left\vert \sum_{i\in \{1,\dots
,N\}^{k}}\varepsilon _{i}(\omega )x_{i_{1}}^{(1)}\cdots
x_{i_{1}}^{(j_{1})}x_{i_{2}}^{(j_{1}+1)}\cdots x_{i_{2}}^{(j_{2})}\cdots
x_{i_{k}}^{(j_{k-1}+1)}\cdots x_{i_{k}}^{(j_{k})}\right\vert ^{2}d\mathbb{P}%
\right) ^{1/2} \\
& \leq \left( \sum_{i\in \{1,\dots ,N\}^{k}}\left\vert x_{i_{1}}^{(1)}\cdots
x_{i_{1}}^{(j_{1})}x_{i_{2}}^{(j_{1}+1)}\cdots x_{i_{2}}^{(j_{2})}\cdots
x_{i_{k}}^{(j_{k-1}+1)}\cdots x_{i_{k}}^{(j_{k})}\right\vert ^{2}\right)
^{1/2} \\
& \leq \Vert \left( x_{i_{1}}^{(j_{1})}\right) _{i_{1}=1}^{N}\Vert
_{2}\cdots \Vert \left( x_{i_{k}}^{(j_{k})}\right) _{i_{k}=1}^{N}\Vert _{2}
\\
& \leq N^{k/2}.
\end{align*}%
Since the $\psi _{2}$-norm of a Rademacher process is dominated by its $L^{2}
$-norm, we get
\begin{equation*}
\Vert X(\cdot ,x)\Vert _{\psi _{2}}\leq CN^{k/2}:=A
\end{equation*}%
for some absolute constant $C>0$.

Now, let $\delta >0$. For a fixed value of $\omega $, for any $x,y$ in $M$
with $\Vert x-y\Vert <\delta $, the multilinearity of $X(\omega ,\cdot )$
ensures that
\begin{equation*}
|X(\omega ,x)-X(\omega ,y)|\leq n\delta \sup_{z\in M}|X(\omega ,z)|.
\end{equation*}%
We set again $\delta =\frac{1}{2n}$ and so%
\begin{equation*}
|X(\omega ,x)-X(\omega ,y)|\leq \frac{1}{2}\sup_{z\in M}|X(\omega ,z)|.
\end{equation*}%
Repeating the previous argument and we observe that there exists a $\delta $%
-net $F$ of $M$ with cardinal less than $\left( 1+\frac{2}{\delta }\right)
^{2nN}=(1+4n)^{2nN}$ (the product $nN$ is the dimension of $(\ell _{\infty
}^{N})^{n}$). Setting $R=\lambda N^{(k+1)/2}$ for some large $\lambda $ (not
depending on $N$ but eventually depending on $n$), we obtain

\begin{equation*}
\frac{\text{card}(F)}{\psi _{2}(R/A)}=\frac{(1+4n)^{2nN}}{e^{\left( \frac{%
\lambda N^{(k+1)/2}}{CN^{k/2}}\right) ^{2}}-1}=\frac{(1+4n)^{2nN}}{e^{\left(
\frac{\lambda ^{2}N}{C^{2}}\right) }-1}<1
\end{equation*}%
and from the lemma there exists $\omega _{0}\in \Omega $ such that, for any $%
x\in M$,
\begin{equation*}
|X(\omega _{0},x)|\leq 2R=2\lambda N^{\frac{k+1}{2}},
\end{equation*}%
i.e.,%
\begin{equation*}
\left\Vert X(\omega _{0},\cdot )\right\Vert \leq 2\lambda N^{\frac{k+1}{2}}.
\end{equation*}%
Now consider an $n$-linear operator $U:c_{0}\times \cdots \times
c_{0}\rightarrow \mathbb{K}$ and if%
\begin{equation*}
\left( \sum\limits_{i_{1},\dots ,i_{m}=1}^{N}\left\vert U\left(
e_{i_{1}}^{n_{1}},\dots ,e_{i_{k}}^{n_{k}}\right) \right\vert ^{r}\right) ^{%
\frac{1}{r}}\leq C\left\Vert U\right\Vert
\end{equation*}%
for all $U$ and all positive integers $N$, where $\left(
e_{i_{1}}^{n_{1}},\dots ,e_{i_{k}}^{n_{k}}\right) $ means $(e_{i_{1}},%
\overset{\text{{\tiny $n_{1}$ times}}}{\dots },e_{i_{1}},\dots ,e_{i_{k}},%
\overset{\text{{\tiny $n_{k}$ times}}}{\dots },e_{i_{k}})$, we have
\begin{align*}
N^{\frac{k}{r}}& =\left( \sum_{i\in \{1,\dots ,N\}^{k}}\left\vert
\varepsilon _{i}(\omega )x_{i_{1}}^{(1)}\cdots
x_{i_{1}}^{(j_{1})}x_{i_{2}}^{(j_{1}+1)}\cdots x_{i_{2}}^{(j_{2})}\cdots
x_{i_{k}}^{(j_{k-1}+1)}\cdots x_{i_{k}}^{(j_{k})}\right\vert ^{r}\right) ^{%
\frac{1}{r}} \\
& \leq C\left\Vert X(\omega _{0},x)\right\Vert  \\
& \leq 2C\lambda N^{\frac{k+1}{2}}.
\end{align*}%
Making $N\rightarrow \infty $ we conclude that $r\geq 2k/(k+1)$. This result
provides the optimality, for instance, of \cite[Corollary 2.5]{uni}. We also
recall that the case $k=n$ recovers the classical
Kahane--Salem--Zygmund--Inequality.

\end{document}